\documentclass[12pt, twoside]{article}
\usepackage[utf8]{inputenc}
\usepackage[T1]{fontenc}
\usepackage{libertineRoman}
\usepackage{amsthm, amsmath, amssymb, amsfonts, enumerate,mathrsfs}
\usepackage[a4paper,margin=3.4cm, headsep=15pt, headheight=3.5cm]{geometry}
\newcommand\smvee{\raise0.3ex\hbox{$\scriptscriptstyle\vee$}}

\usepackage{tikz-cd, tikz}
\usepackage{float}
\usetikzlibrary[patterns, decorations.pathreplacing]
\usetikzlibrary{shapes,trees, positioning} 
\usetikzlibrary{hobby}
\usetikzlibrary{intersections}
\usepackage{caption}
\usepackage{graphicx, relsize}
\usetikzlibrary{matrix, calc}
\usepackage{mathtools}
\usepackage{spectralsequences}

\usepackage{stackengine}
\DeclareUnicodeCharacter{200E}{}

\numberwithin{equation}{subsection}
\newtheorem{theorem}{Theorem}[subsection]

\newtheorem{prop}[subsubsection]{Proposition}

\newtheorem{lemma}[subsubsection]{Lemma}

\newtheorem{defn}[subsection]{Definition}

\newcommand{\myparagraph}{%
	\refstepcounter{subsubsection} 
	\textbf{\thesubsubsection} 
}
\theoremstyle{definition}

\newtheorem{remark}[subsubsection]{Remark}
\usepackage[colorlinks]{hyperref}
\usepackage[nameinlink]{cleveref}
\hypersetup{colorlinks,breaklinks,
	citecolor=[rgb]{0.5,0.0,0.0},
	linkcolor=[rgb]{0.2,0.3,0.8}}

\def\Q{\mathbf{Q}}

\def\mor{\mathrm{Mor}}

\def\jfrak{\mathfrak{j}}
\def\ifrak{\mathfrak{i}}
\def\P{\mathbb{P}}
\def\Pc{\mathcal{P}}

\def\d{\mathbf{d}}
\def\E{\mathscr{E}}
\def\Z{\mathcal{Z}}

\def\F{\mathcal{F}}
\def\I{\mathcal{I}}

\def\X{\mathcal{X}}

\def\Shv{\mathrm{Shv}}
\def\Hom{\mathrm{Hom}}

\def\1{\mathbb{1}}

\def\opn{\mathcal{O}_{\mathbb{P}^N}}
\def\O{\mathcal{O}}
\def\g{\mathbf{g}}

\def\sgn{\mathit{sgn}}
\def\pic{\mathrm{Pic}}

\newcommand\Sym{\mathrm{Sym}}

\date{}

\DeclareMathOperator{\Shom}{\mathscr{H}\text{\kern -3pt {\calligra\large om}}\,}
\DeclareMathOperator{\Sext}{\mathscr{E}\text{\kern -3pt {\calligra\large xt}}\,}

\usepackage{fancyhdr}
\newcommand\shorttitle{{A configuration space model for algebraic function spaces }}
\newcommand\authors{{Oishee Banerjee}}
\usepackage{titlesec}

\titleformat{\section}
{\large\bfseries\scshape}   
{\thesection}           
{1em}                   
{}                      
\titleformat{\subsection}
{\normalsize\bfseries\scshape}   
{\thesubsection}           
{1em}                   
{}      
\titleformat{\subsubsection}
{\normalsize\bfseries}   
{\thesubsubsection}           
{1em}                   
{}     

\title{A configuration space model for algebraic function spaces \\ \vspace{1mm} \textit{\normalsize {To Benson Farb on his {57$^{th}$} birthday}}}

\author{Oishee Banerjee}

\begin{document}
	\maketitle
	
		\begin{abstract}We prove that the space of algebraic maps between two smooth projective varieties, under certain conditions, admit a configuration space model, thereby obtaining an algebro-geometric analogue of Bendersky-Gitler's result (\cite[Theorem 7.1]{BG91}) on topological function spaces. Our result should be a thought of as a natural higher dimensional counterpart of \cite[Theorem 3]{Ban24}.
		
		\end{abstract}

	\section{Introduction}\label{sec:introduction}
	
\paragraph{Motivation.} The study of spaces of \emph{continuous} maps between two topological spaces, and their surprising connections with configuration spaces, has a rich history spanning several decades. Foundational works by Anderson, Bendersky-Gitler, Snaith, Cohen-May-Taylor, Arone, Ahearn-Kuhn, and others have explored this extensively (see \cite{BG91, Anderson72}, and also \cite{AK02, Arone99} and the references therein). One key result is the stable splitting of function spaces under certain constraints (e.g., connectivity conditions on the range), where the components of the splitting include, among other structures, configuration spaces on the domain.

An exact analogue of this phenomenon for algebraic maps between algebraic varieties is unrealistic due to the rigidity of morphisms between varieties. Nevertheless, in this note, we show that the moduli space of algebraic morphisms between two smooth projective varieties, under certain strong conditions on the range, can, in a sense, admit a configuration space model. This establishes an algebro-geometric analogue of Bendersky-Gitler's result on the space of continuous maps between topological spaces (\cite[Theorem 7.1]{BG91}). Our result can be seen as a natural higher-dimensional counterpart of \cite[Theorem 3]{Ban24}.

\paragraph{Setup.} Throughout the paper, we fix smooth projective varieties $X$ and $Y$ over an algebraically closed field of characteristic $0$. We fix $\Upsilon$  a polarization on $Y$ i.e. we fix an embedding $\upsilon: Y\hookrightarrow \P^N$ where $\Upsilon=\upsilon^* \opn(1)$ and $N=\dim |\Upsilon|$, the rank of the complete linear system $|\Upsilon|$. 

A morphism $f: X\to\P^N$ corresponds to a line bundle $L$ on $X$ such that $L = f^*\O_{\P^N}(1)$. Let $\d:= c_1(L) \in N^1(X)$.



A morphism $f:X\to Y$ corresponds to a line bundle $L$ on $X$ such that $L=f^*\Upsilon = {f^*} {\circ} \upsilon^* \opn(1)$. We say $f$ has \emph{degree} $\d$ if $c_1(L) = \d$. Let $\mor_{\d}(X,Y)$ be the moduli space of morphisms $f: X\to Y$ such that $c_1(f^*\Upsilon) =\d$. 


We say a \emph{numerical class $\d\in N^1(X)$ separates $r$ points} if it is ample and if every line bundle $L\in \pic_{\d}(X)$ separates $r$ points\footnote{There are closely related (stronger) notions like $r$-very ampleness, $r$-spannedness etc. which have been studied for decades, pioneered by the work of Beltrametti, Sommese and others, however, the weaker notion of  separating points (as opposed to jets) is sufficient for the purpose of this note. Aumonier in \cite{Aumonier2024} calls this property of line bundles $k$-interpolating. However, `separating points' is a standard phrase in algebraic geometry to describe the phenomenon, so we adhere to that.} (see \ref{lemma:r-separating}). Note that if a line bundle $L$ (respectively, a numerical class $\d$) separates $r$ points, then it separates $(r-1)$ points; let $r(L)$ (respectively, $r(\d)$) denote the maximum number $r$ for which $L$ (respectively, $\d$) separates $r$-points. If $\d \in N^1(X)$ separates at least one point, we define\begin{equation}\label{eq:rd}
	r(\d): = \max \{r: \d \text{ separates } r \text{ points} \}.
\end{equation}
Furthermore, we say $\d$ is \emph{acyclic} of all higher cohomologies of all line bundles in its numerical equivalence class vanish.


 Now we state our theorem.
\begin{theorem}\label{theorem}Let $X$ be a smooth projective variety of dimension $n$, $(Y,\Upsilon)$ a polarized smooth projective variety, and $N:=\dim |\Upsilon|$. Let $\d$ be acyclic, and let $r(\d)$ be as in \eqref{eq:rd}. Then:
	\begin{enumerate}
		\item\label{statement1} 	 Then there exists a first quadrant spectral sequence of Galois representations/mixed Hodge structures:\begin{gather}\label{ss:theorem}
			E_1^{p,q} \implies H^{p+q}_c(\mor_{\d}(X,Y);\Q)
		\end{gather} with \begin{gather}\label{ss:E1terms}
			E_1^{p,*} = (H^*(\mathring{X}^{p};\Q)\otimes \sgn_{S_p})^{S_p}\otimes H^*(\pic_{\d}(X);\Q)\otimes H_c^*(Y(D_{p-1});\Q)
		\end{gather} for all $p\leq r(\d)+1$, where $Y(D_p)$ are certain auxilliary schemes defined in, and satisfying the conditions of \S \ref{para:keyassumption}, and $$\mathring{X}^{p}:= X^p-\{\text{ diagonals}\}.$$
		
		\item\label{statement2}\textbf{Homological stability:} 	In the case when $Y=\P^N$,  the spectral sequence \eqref{ss:theorem} results in \begin{gather}\label{ss:E2terms}
E_2^{p,q}=E_{\infty}^{p,q}
		\end{gather} for all $0\leq p\leq r(\d)+1$, and $\dim(\mor_{\d}(X,Y))-r(\d) \leq q\leq \dim(\mor_{\d}(X,Y))$.
		
		\item\label{statement3} If $\delta:= \d-c_1(K_X)$ is ample, then the stable bound $r(\d)$ satisfies the following equality:\begin{equation}\label{eq:rdbound}
			r(\d) =\left\lfloor\min_{\substack{[W]\in \mathrm{CH}_k(X),\\ 1\leq k\leq n}} \frac{2(\delta^k.[W])^{\frac{1}{k}}-n^2+n-1}{2}\right\rfloor -1
		\end{equation} where $\mathrm{CH}_k(X)$ is the $k^{th}$ Chow group of $X$.
	\end{enumerate}
\end{theorem}
	
\paragraph{Some remarks and context} \begin{enumerate}
	\item \textbf{Poincare/Koszul duality.} A recurring phenomenon in the study of such algebraic function spaces, as demonstrated in this note, as well as in \cite{Ban24}, is the appearance of a Koszul-type cochain complex (see \eqref{isom:twistedbysgn}) in the analysis of the Poincare dual of $\mor_{\d}(X,Y)$. Whereas it appears naturally in our proof via the theory of hypercovers, it does beg the question of whether there is a factorization homology approach to proving something like Theorem \ref{theorem}. 
	
	An affirmative answer has already been provided by Ho (\cite{Ho20}) in the case $X=\P^1$ and $Y=\P^n$ (because in such cases, algebraic maps simply boil down to studying zero-cycles on $X$) and there is paramount evidence of an affirmative answer for curves of higher genera by exploiting algebraic non-abelian Poincare duality in the sense of Gaitsgory-Lurie (\cite[\S 3]{GL17}). However, a naive translation of the curve-case to higher dimensional domains have several obvious pitfalls, some of which becomes clear as we go through the proof, and discussing the rest would take us too far afield (see e.g. \cite[\S 1.5]{GL19}).

	\item \textbf{$\mor_{\d}(X,Y)$ vs. its topological analogue $\mathrm{Top}_{\d}(X,Y)$.} In a direction distinct from the study of the relationship between (continuous) function spaces and configuration spaces, significant attention has been given to the comparison of spaces of holomorphic/algebraic functions with that of continuous functions between complex holomorphic manifolds since the '70s (for a brief history of it see \cite[\S 1.1]{Aumonier2024}). Two notable examples are: Mostovoy's work (\cite{Mostovoy06}) where $X,Y$ are both projective spaces, and a recent work of Aumonier's (see \cite{Aumonier2024}) where he compares stable homology of $\mor_{\d}(X,\P^N)$ to that of continuous maps $X \to \P^N$ for arbitrary smooth projective $X$, thus significantly extending Mostovoy's result. 
	
	If our theorem relating the (cohomology of the) space of algebraic maps to configuration spaces appears somewhat unexpected, it is worth highlighting that results by Anderson, Bendersky-Gitler (\cite{Anderson72, BG91}), when combined with those of Aumonier or Mostovoy, lend credence to such a connection.
	
Our theorem not only confirms this connection but also strengthens it in two key ways: (a) it demonstrates that the Galois representations or mixed Hodge structures on both sides are preserved, and (b) it underscores the role of the intersection theory of $X$, as evidenced by Equation \ref{eq:rdbound}. Notably, pulling back a 'configuration space model' for the space of continuous maps via Segal-type results from Aumonier or Mostovoy offers no insight into Hodge structures. In contrast, our approach—grounded entirely in algebraic geometry—makes these structures explicit and central to the discussion.
	
	The only prior instance of an explicit comparison between the space of algebraic function spaces and configuration spaces appears in the author's earlier work \cite{Ban24, Banerjee2022}--- those were in the case when $X$ is a curve. 
	
	\item  \text{Adjoint line bundles and point-separation.} The property of separating finitely many points is not  numerical property of a line bundle $L$.  They are, however, numerical properties on the \emph{adjoint} line bundle $L\otimes K_X$ (or, more generally, on $L\otimes K_X^{\otimes m}$ for suitable values of $m$). The interested reader may refer to the works of Angehrn-Siu, Reider, Demailly, Ein-Lazarsfeld-Nakamaye, Kollar etc (see \cite{AS95} and the references therein). Which is why explicit bounds like \eqref{eq:rdbound} is available only for adjoints of line bundles, not the line bundles themselves. 
	
	\item \textbf{On the auxilliary schemes $Y(D_p)$.} The condition of the `auxilliary schemes' $Y(D_p)$ being of \emph{Leray-Hirsch type} (see paragraph \ref{para:keyassumption}) over $\pic_{\d}(X)$ in a range of values of $p$, is absolutely indispensable. For a general $Y$, the schemes $Y(D_p)$ can often be empty, and in the cases when they are non-empty, proving its non-emtpyness is usually highly nontrivial. Case at hand is the geometric Manin's conjecture: even in the seemingly simple case of $X=\P^1$, if $Y$ is a low degree hypersurface in a sufficiently high dimensional projective space, the proof is extremely difficult--- as shown by Browning-Sawin in \cite{BS20}. See \S\ref{para:keyassumption} for a further discussion on the terminology and the importance of the additional condition of being Leray-Hirsch type.
	
	\item \textbf{Divergence of $r(\d)$.} Whereas the range $Y$ often poses insurmountable difficulties, tackling an arbitrary domain is relatively simpler, at least in the context of questions like stable homology of these function spaces--- it only needs a part of the intersection theory of $X$ as an input. Observe that \eqref{eq:rdbound} for the stable bound depends solely on the intersection theory of $X$. In particular, higher the positivity of $\delta$, higher the intersection numbers $\delta^k.[W]$, and since the positivity of $\delta$ diverges, so does $r(\d)$.
	
	\item \textbf{Spaces of sections of vector bundles.} Observing that sections of vector bundles on a smooth projective variety $X$ can themselves be viewed as algebraic function spaces of a specific kind (locally maps to $\mathbb{A}^N$ on $X$), a careful reader following our proof would notice that our methods translate almost verbatim—indeed, more straightforwardly in the absence of geometric complexities from $Y$ or $\pic_{\d}(X)$—to providing an alternative proof of the cohomological results by Das and Howe (see \cite{DH24}). One only needs to replace the condition of vector bundles \emph{separating points} by various degrees fof \emph{jet-ampleness}. While the author does not claim expertise in their approach, it seems plausible that Das and Howe's cohomological inclusion-exclusion framework could, in principle, be adapted to derive results analogous to our Theorem \ref{theorem}.
	
\end{enumerate}

\noindent	\textbf{Method of Proof.} Our proof is sheaf-theoretic, morally similar to \cite{Ban24, Banerjee2022}. However, due to the oft-encountered dichotomy between the algebro-geometric complexity of varieties of dimension $1$ and higher, we need to incorporate modern perspectives on certain classical results.

A fundamental component of our approach is the use of cohomological descent for proper hypercoverings. We construct a natural compactification of $\mor_{\d}(X, Y)$, and a proper hypercover (\eqref{def:Xr}) augmented on its Poincare dual that admits cohomological descent. At the level of sheaves, this hypercover--- via a the passage through the \emph{symmetric simplicial category},  originally introduced in the context of additive K-theory by Feigin and Tsygan (\cite{FT87}), and later further developed by Fiederowicz and Loday (\cite{FL91})--- gives rise to a Koszul-type complex (see \eqref{isom:twistedbysgn}) whose cohomology produce \eqref{spectralsequence}. 


A minor technical note:  we work in the derived $\infty$-category of constructible sheaves with coefficients in $\Q$-vector spaces, as developed in \cite{GL19, Lurie17}, adopting the framework of six-functor formalism by Liu-Zheng \cite[\S 9.3]{LZ24}. While we employ $\infty$-categorical formalism, it is largely cosmetic—most of the proof of Theorem \ref{theorem} translates seamlessly into the language of the triangulated derived category of constructible sheaves. The only caveat is the need to analyze group invariants of objects that traditionally reside in traingulated derived category of constructible sheaves. The $\infty$-categorical language offers a cleaner presentation of our core ideas, avoiding technical distractions that are well-established in the literature.

Our estimate of $r(\d)$ in \eqref{eq:rdbound} follows directly from Angehrn-Siu's work on a conjecture of Fujita (see \cite[Theorem 0.1]{AS95}).

\paragraph{Acknowledgement.} \textnormal{My deepest thanks to Robert Lazarsfeld for sharing his expertise on positivity properties of line bundles.}

\section{The space of morphisms from $X$ to $Y$} To clearly convey the key ideas, we begin by developing our framework for the special case where $Y=\P^N$.
In \S\ref{subsection:morXY} we will generalize these methods to the case of an arbitrary polarized smooth projective variety $Y$. Throughout this section, $X$ is a smooth projective variety of dimension $n$ over a fixed algebraically closed field of characteristic $0$.
 
\subsection{Preliminaries on positivity of line bundles}
Positivity of line bundles can be studied from various angles: from the concept of separation of points, jets at points, cohomological vanishing theorems, etc, and they are all deeply intertwined. In this subsection we record some definitions and facts pertaining to Theorem \ref{theorem}.

 \begin{defn}\label{def:r-separating}
Let $r$ be a positive integer. We say a line bundle $L$ separates $r$ points if it is ample and its global sections separate any set for $r$ points in $X$ i.e. for all $Z\subset X$ with $|Z|=r$, one has \[H^1(L\otimes \mathcal{I}_Z)=0\] or equivalently, if the restriction map 
\[H^0(X,L)\twoheadrightarrow H^0(X,L\otimes \mathcal{O}_Z)\] is surjective, 
where $\mathcal{I}_Z$ and $\mathcal{O}_Z$ denote the ideal sheaf and structure sheaf of $Z$ respectively.

We say a numerical class $\d\in N^1(X)$ separates $r$ points if it is ample and every line bundle in $L\in\pic_{\d}(X)$ separates $r$ points.
	\end{defn}

	Let us now record a theorem by Angehrn-Siu (see \cite[Theorem 0.3]{AS95}), suitably paraphrased to meet our requirements:\begin{theorem}\label{theorem:AS}
Let $r$ be a positive integer. If \begin{equation}\label{ineq:AS}
(L^k \cdot [W])^{\frac{1}{k}}>\frac{1}{2}n(n+2r-1)
\end{equation}
for any irreducible subvariety $W$ of dimension $1 \leq  k \leq n=\dim X$ in $X$, then the
global holomorphic sections of $L \otimes K_X$ over $X$ separate any set of $r$ distinct
points of $X$.
	\end{theorem}
	Observe that by \cite[Proposition 20.1.4]{Vakil15}, the intersection number in the inequality \eqref{ineq:AS} depends only on the numerical equivalence class of $L$. 
	
\begin{defn}
	We say a line bundle $L$ is \emph{acyclic} is $H^i(X,L)=0$ for all $i>0$. We a numerical equivalence class $\d\in N^1(X)$ is acyclic if $H^i(X,L)=0$ for all $i>0$ and all $L\in \pic_{\d}(X)$.
\end{defn}
\noindent Note that acyclicity is not a numerical property. We now prove a simple but important lemma.
	\begin{lemma}\label{lemma:r-separating}
	Let $r$ be a positive integer.	Then there exists $\d\in N^1(X)$ such that $\d$ separates $r$ points and is acyclic.
	\end{lemma}
	\begin{proof}
	Pick $\delta \in N^1(X)$ ample. Tensoring up if necessary, and in view of the observation that \eqref{ineq:AS} depends only on the numerical equivalence class of $L$, we can ensure that $\delta$ satisfies the inequality \eqref{ineq:AS} on the intersection numbers in Theorem \ref{theorem:AS}.
		
		
		Because $\delta$ is ample, by the Kodaira vanishing theorem (\cite{Kodaira53}), for all $L\in \pic_{\delta}(X)$ one has $H^i(X,K_X\otimes L)=0$ for all  $i>0$. Furthermore, since $\delta$ has been chosen to satisfy the inequality \eqref{ineq:AS}, $L\otimes K_X$ separates $r$ points by Theorem \ref{theorem:AS}, for all $L\in \pic_{\delta}(X)$. Then the numerical equivalence class given by $\d :=c_1(K_X)+\delta$ satisfies the conclusion of the lemma.
	\end{proof}
	
		\noindent\myparagraph For the rest of the paper we fix $\delta\in N^1(X)$ ample, such that $\d=c_1(K_X)+\delta$ is acyclic and separates $r$ points (we know such a $\d$ exists by Lemma \ref{lemma:r-separating}), and denote by $r(\d)$ the maximum number of points $\d$ separates i.e. the maximum $r$ for which $\d$ separates $r$ points. By Angehrn-Siu's estimates (see \eqref{ineq:AS}), we have \begin{equation}\label{eq:rdestimate}
		r(\d) =\left\lfloor\min_{\substack{[W]\in \mathrm{CH}_k(X),\\ 1\leq k\leq n}} \frac{2(\delta^k.[W])^{\frac{1}{k}}-n^2+n-1}{2}\right\rfloor .
	\end{equation}
	
	\subsection{Geometry and topology of $\mor_{\d}(X,\P^N)$} In this subsection we construct a natural compactification of $\mor_{\d}(X,\P^N)$, and a hypercover augmented over its `discriminant locus'.
	
	\noindent \myparagraph Let $\d$ be as above. The space $\mor_{\d}(X,\P^N)$ is given by: 
		\begin{align}
\mor_{\d}(X,\P^N)=\bigg\{(L,[s_0:\ldots : s_N]): L\in \pic_{\d}(X), s_i\in \Gamma(X,L), 0\leq i\leq N, \nonumber\\ \bigcap_{0\leq i\leq N}\mathrm{div}(s_i)=\emptyset\bigg\}
		\end{align}
	
\noindent	Define the following space, which we dub as the `discriminant locus': \begin{align}\label{def:discriminant}
\Z_{\d}(X,\P^N):= \bigg\{(L,[s_0:\ldots: s_N]): L\in \pic_{\d}(X), s_i\in \Gamma(X,L), 0\leq i\leq N, \nonumber\\ \bigcap_{0\leq i\leq N}\mathrm{div}(s_i)\neq\emptyset\bigg\}.
	\end{align} Whereas morphisms $X\to \P^N$ are given by basepoint free $(N+1)$-tuples of global sections of line bundles in $\pic_{\d}(X)$, the space $\Z_{\d}(X,\P^N)$ is given by $(N+1)$-tuples of global sections that vanish simultaneously at some point in $X$. 
	
\subsubsection{Compactification of $\mor_{\d}(X,\P^N)$}\label{para:Poincare}	Let $\Pc(\d)$ denote the Poincare bundle on $X\times \pic_{\d}(X)$; denoting by \begin{equation}\label{map:pushforwardPoincare}
		\nu:X\times \pic_{\d}(X)\to \pic_{\d}(X)
	\end{equation} the projection to the second factor $\pic_{\d}(X)$. Note that $\d$ being acyclic implies $\nu_*\Pc(\d)$ is a locally free sheaf, and equivalently, a vector bundle on $\pic_{\d}(X)$, of rank 
	\begin{gather}\label{eq:N_d} N_d:= \int_X \mathrm{ch}(L)\,\, \mathrm{td}(X)\end{gather} 
	given by the Hirzebruch-Riemann-Roch theorem (see \cite{Hirzebruch1978}; note that $\mathrm{ch}(L)$ depends only on $\d$, and $td(X)$, the Todd class of $X$, is a geometric property of $X$). The fibre of $\nu_*\Pc(\d)^{\oplus(N+1)}$ over a point $L\in\pic_{\d}(X)$ is $\Gamma(X,L)$. In turn, we can now see that there is a natural open immersion \begin{gather}\label{map:openimmersion}
\jfrak: \mor_{\d}(X,\P^N) \hookrightarrow \P(\nu_*\Pc(\d)^{\oplus(N+1)})
	\end{gather} the latter being the (relative) projectivization of the vector bundle $\nu_*(\Pc(\d))^{\oplus (N+1)}$ over $\pic_{\d}(X)$, and the complement of (the image of) $\mor_{\d}(X,\P^N)$  in $\P(\nu_*\Pc(\d)^{\oplus(N+1)})$ is $\Z_{\d}(X,\P^N)$. Let us denote $\P(\nu_*\Pc(\d)^{\oplus(N+1)})$ by $\X_{-1}$ henceforth; in other words, the points of $\X_{-1}$ are described by:\begin{gather}\label{def:X_{-1}}
\X_{-1}=\bigg\{(L,[s_0:\ldots: s_N]): L\in \pic_{\d}(X), s_i\in \Gamma(X,L), 0\leq i\leq N\bigg\}.
	\end{gather}

\noindent\myparagraph\textbf{Hypercover over $\Z_{\d}(X,\P^N)$} For each $r\geq 0$, define the following spaces:
	\begin{align}\label{def:Xr}
		\X_r:= \bigg\{\big((L,[s_0:\ldots: s_N]), (x_0,\ldots, x_r)\big): L\in \pic_{\d}(X), s_i\in \Gamma(X,L), 0\leq i\leq N, \nonumber\\ x_j\in \bigcap_{0\leq i\leq N}\mathrm{div}(s_i), \text{ for all }0\leq j\leq r\bigg\}\nonumber\\\subset X^{r+1}\times \Z_{\d}(X,\P^N)
	\end{align} We now aim to understand the geometry of the spaces $\X_r$. For starters, we prove the following:
	\begin{lemma}\label{lemma:X0}
		The space $\X_0$ is a smooth projective variety.
	\end{lemma}
	\begin{proof}[Proof of Lemma \ref{lemma:X0}]
Let \begin{gather}
\mathrm{pr}_{23}:X\times X\times \pic_{\d}(X)\to X\times \pic_{\d}(X)
\end{gather} be the projection to the last two factors, \begin{gather}
\mathrm{pr}_{13}:X\times X\times \pic_{\d}(X)\to X\times \pic_{\d}(X)
\end{gather} be the projection to the first and third factors, and \begin{gather}
\mathrm{pr}_{23}:X\times X\times \pic_{\d}(X)\to X\times X
\end{gather} be the projection to the first two factors. Let $D_0(X)$ denote the diagonal in $X\times X$. Then \begin{gather}\label{edf:E0}
\E_0 := (\mathrm{pr}_{13})_{*}(\mathrm{pr}^*_{12}\mathcal{I}_{{D_0}(X)}\otimes \mathrm{pr}^*_{23}\Pc(\d)).
\end{gather} where  $\mathcal{I}_{{D_0}(X)}$ is the ideal sheaf of the diagonal $D_0(X)$. That $\E_0$ is a locally free sheaf is a simple consequence of Grauert's base change theorem (whose proof is complex analytic) or its algebraic rendition by Grothendieck (see \cite[Theorem 12.11]{Harshorne77} or \cite[EGA III 7.7]{Grothendieck63}). Indeed, first note that $\mathrm{pr}_{13}$ is a flat projective morphism, and $D_0(X)\times \pic_{\d}(X)$ is flat over $X\times \pic_{\d}(X)$. Then, letting $$\F:=\mathrm{pr}^*_{12}\mathcal{I}_{{D_0}(X)}\otimes \mathrm{pr}^*_{23}\Pc(\d)$$  note that \begin{gather}
R^i{\mathrm{pr}_{13}}_*\F=0 \text{ for all } i>0
\end{gather} because its stalks are $H^i(X,L\otimes \I_{x})$ for a point $(x,L)\in X\times \pic_{\d}(X)$. Now, for any finite subset $Z\subset X$ with $\# Z\leq r$, consider the short exact sequence of coherent sheaves \begin{gather}\label{seq:I_Z}
0\to I_Z\to \mathcal{O}_X\to \O_Z\to 0,
\end{gather}tensor with $L$ throughout and take cohomology to obtain \begin{gather}\label{leq:I_ZL}
0\to H^0(X,L\otimes \I_Z)\to H^0(X,L)\to H^0(X,L\otimes \O_Z)\to \cdots \\\cdots\to H^i(X,L\otimes \I_Z)\to H^i(X,L) \to H^i(X,L\otimes \O_Z)\to \cdots\nonumber
\end{gather}Noting that $H^i(X,L\otimes \O_Z)=0$ for all $i>0$, ($L\otimes \O_Z$ are skyscraper sheaves), we obtain $$H^i(X,L\otimes \I_Z)\xrightarrow{\cong} H^i(X,L)=0$$ for all $i>0$, where the latter isomorphism to $0$ is because $\d$ is acyclic. Taking $r=1$ we obtain $\E_0$ is locally free over $X\times \pic_{\d}(X)$: the fibre of the corresponding vector bundle at a point $(x,L)\in X\times \pic_{\d}(X)$ is $\Gamma(X,L\otimes \I_x)$, i.e. global sections of $L$ that vanish at $x$.

Finally, the observation that $\X_0$ is, by definition,  the projectivization of the vector bundle $\bigoplus_{i=0}^{N}\E_0$ over $X\times\pic_{\d}(X)$, concludes the proof of the lemma.

\end{proof}

\noindent Thanks to Lemma \ref{lemma:X0}, the natural map\begin{gather}
\X_0\to \X_{-1}\nonumber \\ \big((L,[s_0:\ldots: s_N]), x\big)\mapsto (L,[s_0:\ldots: s_N])
\end{gather}  is projective, and hence proper; furthermore, its image is clearly $\Z_{\d}(X,\P^N)$, and $$\X_0\twoheadrightarrow \Z_{\d}(X,\P^N)$$ is birational---because a generic element in $\Z_{\d}(X,\P^N)$ is an $(N+1)$-tuple of sections in $\Gamma(X,L)$ for some $L\in \pic_{\d}(X)$ that have exactly one common zero, which, in turn, is because by our assumption $r(\d)\geq 1$, i.e. $L$ is very ample, and one can always choose $(N+1)$ hyperplane sections of the embedding of $X$ by the complete linear system $|L|$ such that those $(N+1)$ hyperplanes intersect at exactly one point--- and since $\X_0$ is smooth, we have now concluded the following:
\begin{lemma}\label{lemma:ResolutionOfSingularities}
The natural map $\X_0 \to \Z_{\d}(X,\P^N)$ is a resolution of singularities.
\end{lemma} 

\noindent Observe that for all $r\geq 0$, we have, by its very definition, 
\begin{gather}\label{def:XrFibreProduct}
\X_r= \underbrace{\X_0\times_{\X_{-1}} \X_0 \times_{\X_{-1}}\times \cdots\times_{\X_{-1}}\X_0}_{r+1}, 
\end{gather} and that there are natural maps \begin{gather}\label{map:pi_r}
\pi_r: \X_r\to \Z_{\d}(X,\P^N) \\ \big((L,[s_0:\cdots: s_N]), (x_0,\cdots, x_r)\big)\mapsto (L,[s_0:\cdots: s_N]). \nonumber
\end{gather}
Recalling the standard fact that proper maps are stable under base change \cite[Corollary 4.8]{Harshorne77}, paired with our earlier observation that $\X_0\to \X_{-1}$ is projective, and hence proper, we come to the following conclusion:

\begin{lemma}\label{lemma:hypercover}
The simplicial scheme $\pi_{\bullet}:\X_{\bullet} \to \Z_{\d}(X,\P^N)$ is a proper hypercover augmented over $\Z_{\d}(X,\P^N)$, with the $i^{th}$ face maps $\X_r\to \X_{r-1}$ given by forgetting the $i^{th}$-factor from the expression \eqref{def:XrFibreProduct}, and the degeneracy maps given by the diagonal embeddings.
\end{lemma} 
\noindent 	The proof of Lemma \ref{lemma:X0} translates almost verbatim to give us the geometry of $\X_r$ for any $r\leq r(\d)$. For $0<r\leq r(\d)$, the schemes $\X_r$ are not smooth, however, they are naturally projectivizations of coherent sheaves that are locally free over the locally closed strata provided by the diagonals in $X^r$. Indeed, let \begin{gather}
	\mathrm{pr}_{23}:X^{r+1}\times X\times \pic_{\d}(X)\to X\times \pic_{\d}(X)
\end{gather} be the projection to the last two factors, \begin{gather}
	\mathrm{pr}_{13}:X^{r+1}\times X\times \pic_{\d}(X)\to X^{r+1}\times \pic_{\d}(X)
\end{gather} be the projection to the first and third factors, and \begin{gather}
	\mathrm{pr}_{23}:X^{r+1}\times X\times \pic_{\d}(X)\to X^{r+1}\times X
\end{gather} be the projection to the first two factors, and let $D_r(X)\subset X^{r+1}\times X$ be the closed subscheme given by \[D_r(X):= \bigg\{\big((x_0,\cdots, x_r),x\big): x=x_i \text{ for some } 0\leq i\leq r\bigg\}\] Then define \begin{gather}\label{edf:Er}
\E_r := (\mathrm{pr}_{13})_{*}(\mathrm{pr}^*_{12}\mathcal{I}_{{D_0}(X)}\otimes \mathrm{pr}^*_{23}\Pc(\d)).
\end{gather} where  $\mathcal{I}_{{D_r}(X)}$ is the ideal sheaf of $D_r(X)\subset X^{r+1}\times X$. Then, following through the proof of Lemma \ref{lemma:X0} and invoking the sequences \eqref{seq:I_Z} and \eqref{leq:I_ZL}, we see that $\E_r$ is locally free on each of the locally closed strata provided by the intersection of the diagonals, its stalk at a point $$\big((x_0,\cdots, x_r), L\big)\in X^{r+1}\times\pic_{\d}(X)$$ (note that $x_i$s need not be distinct) is $$\Gamma(X,L\otimes \I_{\mathrm{Supp(\sum_i x_i)}})$$ where $\mathrm{Supp}(\sum_i x_i)$ is the (set-theoretic) support of the $0$-cycle $\sum_i x_i$ corresponding the $(r+1)$-tuple $(x_0,\cdots, x_r)$ of points in $X^{r+1}$; in turn, if $X^{r+1}_{(m)}\subset X^{r+1}$ is defined as the locus of $(r+1)$-tuples of points of which exactly $m$ are distinct, then we have the following: \begin{lemma}\label{lemma:Xr}
For all $0\leq r\leq r(\d)$, we have \begin{gather}
\X_r=\underline{\mathrm{Proj}}_{X^{r+1}\times \pic_{\d}(X)}\Sym(\E_r^{\oplus(N+1)})
\end{gather}where $$\E_r\to X^{r+1}\times \pic_{\d}(X)$$ is a stratified vector bundle: for $1\leq m\leq r+1$, on each locally closed strata $X_{(m)}^{r+1}\times \pic_{\d}(X)$, it is a vector bundle of rank $(N_d-m)(N+1)$.
\end{lemma}We let $\mathring{X}^{r+1}$ denote $X_{(r+1)}^{r+1} = {X}^{r+1}-\text{ diagonals}$.
\begin{remark}
	The notion of a stratified vector bundle has several approaches: all minor variants of each other, depending on what one's end goal is. A modern reference which speaks of its development and recent applications is by Ross (\cite{Ross2024}). In our case, the meaning has been made clear by specifying the strata explicitly, which does not require any outside knowledge on the theory of stratified vector bundles. The interested reader can refer to \cite{Ross2024} and the references therein.
\end{remark}
	
	\subsection{Geometry and topology of $\mor_{\d}(X,Y)$}\label{subsection:morXY}
	Recall, from the introduction, that we fixed $Y$ to be a polarized smooth projective variety, $\Upsilon$  the polarization on $Y$ i.e. we fix an embedding $\upsilon: Y\hookrightarrow \P(\Gamma(Y,\Upsilon)^{*})\cong \P^N$, where $\Upsilon=\upsilon^* \opn(1)$ and $N=\dim |\Upsilon|$, the rank of the complete linear system $|\Upsilon|$. Furthermore, let $\g=(g_1,\cdots, g_m)$ denote the set of generators of the homogenous ideal of $Y$ in $\Sym(\Gamma(Y,\Upsilon))$.
	
\subsubsection{Compactification of $\mor_{\d}(X,Y)$} Now we construct a natural compactification of $\mor_{\d}(X,Y)$, and a hypercover augmented over its `discriminant locus'. Let $\d\in N^1(X)$ be as in Lemma \ref{lemma:r-separating}.  The space $\mor_{\d}(X,Y)$ is given by: 
\begin{align}
	\mor_{\d}(X,Y)=\bigg\{(L,[s_0:\ldots : s_N]): L\in \pic_{\d}(X), s_i\in \Gamma(X,L), 0\leq i\leq N, \nonumber\\ \bigcap_{0\leq i\leq N}\mathrm{div}(s_i)=\emptyset, \,\, \g(s_0,\cdots, s_N)=0\bigg\}
\end{align}\noindent Thus, $\mor_{\d}(X,Y)$ is naturally a closed (not necessarily nonempty) subscheme of $\mor_{\d}(X,\P^N)$. 
As before, we define the `discriminant locus' as: \begin{align}\label{def:discriminantY}
	\Z_{\d}(X,Y):= \bigg\{(L,[s_0:\ldots: s_N]): L\in \pic_{\d}(X), s_i\in \Gamma(X,L), 0\leq i\leq N, \nonumber\\ \bigcap_{0\leq i\leq N}\mathrm{div}(s_i)\neq\emptyset,  \,\, \g(s_0,\cdots, s_N)=0\bigg\}.
\end{align} 

\subsubsection{Hypercover over $\Z_{\d}(X,Y)$} For each $r\geq 0$, we define the following spaces:
\begin{align}\label{def:XrY}
	\X_r(Y):= \bigg\{\big((L,[s_0:\ldots: s_N]), (x_0,\ldots, x_r)\big): L\in \pic_{\d}(X), s_i\in \Gamma(X,L), 0\leq i\leq N, \nonumber\\ x_j\in \bigcap_{0\leq i\leq N}\mathrm{div}(s_i), \text{ for all }0\leq j\leq r, \,\, \g(s_0,\cdots, s_N)=0\bigg\}\nonumber\\\subset X^{r+1}\times \Z_{\d}(X,Y)
\end{align}At this point, we refer back to paragraph \ref{para:Poincare}. The natural map in \ref{map:openimmersion}: $$\mor_{\d}(X,\P^N)\hookrightarrow  \P(\nu_*\Pc(\d)^{\oplus(N+1)})$$ was shown to be an open immersion. We denote, by $\X_{-1}(Y)$, the following closed subscheme of $\P(\nu_*\Pc(\d)^{\oplus(N+1)})$---\begin{gather}
\X_{-1}(Y)= \bigg\{(L,[s_0:\ldots: s_N]): L\in \pic_{\d}(X), s_i\in \Gamma(X,L), 0\leq i\leq N, \nonumber\\ \g(s_0,\cdots, s_N)=0\bigg\}
\end{gather}Observe that when $Y=\P^N$ we get back our old definitions of $\X_r$, $r\geq -1$. Observe that we have a natural open immersion \begin{gather}\label{map:jfrak}
\jfrak: \mor_{\d}(X,Y)\hookrightarrow\X_{-1}(Y)
\end{gather} and its closed complement is $\Z_{\d}(X,Y)$; we let \begin{gather}\label{map:ifrak}
\ifrak: \Z_{\d}(X,Y)\hookrightarrow\X_{-1}(Y)
\end{gather}denote the closed inclusion.

 For an arbitrary $Y$, the space $\X_0(Y)$ as defined above is hardly ever smooth, so the statement of Lemma \ref{lemma:X0} no longer holds. However, a part of the proof of Lemma \ref{lemma:X0} clearly translates to the case of an arbitrary $Y$ to show that $\X_0(Y)$ is a projective variety. Indeed, in the proof of Lemma  \ref{lemma:X0} we saw that $\X_0$ is the projectivization of the vector bundle
$\nu_*(\E_0)$ over $X\times\pic_{\d}$. By definition of $\X_0(Y)$ we have the following:
\[
\begin{tikzcd}
	\X_0(Y) \arrow[rr, hook] \arrow[dr] && \X_0 \arrow[dl] \\
	& X\times \pic_{\d}(X)
\end{tikzcd}
\] where the horizontal arrow is a closed embedding, the fibres of $\X_0(Y)$ over a point $(x,L)\in X\times \pic_{\d}(X)$ is given by the locus of $[s_0:\cdots: s_N]$ such that $$\g(s_0:\cdots:s_N)=0,$$ where $s_0,\ldots, s_N\in \Gamma(X,L\otimes \I_x)$. By a similar reasoning $\X_{-1}(Y)$ is also a projective variety: \[
\begin{tikzcd}
	\X_{-1}(Y) \arrow[rr, hook] \arrow[dr] && \X_{-1}\arrow[dl] \\
	& \mathrm{Pic}_{\d}(X)
\end{tikzcd}
\] where the horizontal arrow is a closed embedding, the fibres of $\X_{-1}(Y)$ over a point $L\in  \pic_{\d}(X)$ is given by the locus of $[s_0:\cdots: s_N]$ such that $\g(s_0:\cdots:s_N)=0,$ where $s_0,\ldots, s_N\in \Gamma(X,L)$.  The natural map \begin{gather}
\X_0(Y)\twoheadrightarrow \Z_{\d}(X,Y)\subset \X_{-1}(Y)\nonumber\\(L,[s_0:\cdots: s_N]), x\mapsto (L,[s_0:\cdots: s_N])
\end{gather}is projective, and hence proper. Also, for all $r\geq 0$, we have, by its very definition, 
\begin{gather}\label{def:XrFibreProductY}
	\X_r(Y)= \underbrace{\X_0(Y)\times_{\X_{-1}(Y)} \X_0(Y) \times_{\X_{-1}(Y)}\times \cdots\times_{\X_{-1}(Y)}\X_0(Y)}_{r+1}, 
\end{gather} and there are natural maps, as one would expect,\begin{gather}\label{map:pi_rY}
	\pi_r: \X_r(Y)\to \Z_{\d}(X,Y) 
\end{gather}given by composing the face maps, the $i^{th}$ face map given by forgetting the $i^{th}$ factor: $$	\X_r(Y)\to 	\X_{r-1}(Y).$$
Recalling the standard fact that proper maps are stable under base change \cite[Corollary 4.8]{Harshorne77}, paired with our earlier observation that $\X_0(Y)\to \X_{-1}(Y)$ is projective, and hence proper, we obtain a generalization of Lemma \ref{lemma:hypercover}:

\begin{prop}\label{lemma:hypercoverY}
	The simplicial scheme $\pi_{\bullet}:\X_{\bullet}(Y) \to \Z_{\d}(X,Y)$ is a proper hypercover augmented over $\Z_{\d}(X,Y)$, with the $i^{th}$ face maps $\X_r(Y)\to \X_{r-1}(Y)$ given by forgetting the $i^{th}$-factor from the expression \eqref{def:XrFibreProductY}, and the degeneracy maps given by the diagonal embeddings.
\end{prop} 

\noindent\myparagraph\label{para:keyassumption} \textbf{Leray-Hirsch type.} As noted earlier, for an arbitrary $Y$ there is no guarantee the schemes $\mor_{\d}(X,Y)$ or $\X_r(Y)$ are nonempty, and even if they are, there is no general recipe to understand their geometry. Note that for $Y=\P^N$, Lemma \ref{lemma:Xr} not only implies $\X_r$ is a stratified fibre bundle in a range of values of $r$, but being projectivization of vector bundles on those strata, it satisfies the Leray-Hirsch theorem.

Throughout this paper we work on basis of the assumption that  $\mor_{\d}(X,Y)$ is  nonempty, and that $\X_r(Y)$, for all $r\leq r(\d)$, is nonempty, and satisfies $\Q$-Leray-Hirsch (i.e. it satisfies the Leray-Hirsch theorem with coefficients in $\Q$) on each locally closed strata $X_{(m)}^{r+1}\times\pic_{\d}(X)\subset X^{r+1}\times\pic_{\d}(X)$ for $1\leq m\leq r(\d)$. Let $Y(D_r)$ denote the fibre of $\X_r(Y)$ over $\mathring{X}^{r+1}\times\pic_{\d}(X)$ (where, recall that $\mathring{X}^{r+1}= X^{r+1}-\text{ diagonals}$).

\begin{remark}
	There are no sufficient conditions one can impose on a reasonably large class of $Y$s to result in $\X_r(Y)$  carrying the structure of a stratified fibre bundle a la Lemma \ref{lemma:Xr}, let alone satisfy Leray-Hirsch. An analogue of Lemma \ref{lemma:Xr} is not hard to prove when $Y$ carries a nice universal torsor description (see the proof of \cite[Theorem 1]{Banerjee2022}, and also works on Manin's conjecture in \cite{BT98}, \cite{Pieropan16} and the references therein). However, for a general $Y$, a proof of such a fact, when true, is largely elusive, and is deeply rooted in the theory of algebraic cycles in $Y$--- for a detailed discussion on algebraic fibre bundles being of \emph{Leray-Hirsch} type, and its relation to Hodge conjecture, see Meng's work in \cite{Meng2021}.
\end{remark}

\section{Symmetric (co)simplicial sheaves}
By Lemma \ref{lemma:hypercover}, the hypercover defined in \eqref{def:XrFibreProduct} admits cohomological descent--- which is really the main ingredient for proving Theorem \ref{theorem}.  However, we study the hypercover \eqref{def:XrFibreProduct} as a $\Delta S$-scheme, as opposed to a simplicial scheme, where $\Delta S$ denotes \emph{symmetric simplicial category}. The advantage of working over $\Delta S$, instead of the standard simplicial category $\Delta$, is that the second statement of Theorem \ref{theorem} almost comes for free; however, working over $\Delta S$ results in considering (co)invariants under group actions in the derived category of constructible sheaves, an operation that do not behave well in triangulated categories. 

In this section, we collect the necessary facts about $\Delta S$, and then overcome the obstacle of working in the derived category of constructible sheaves by rephrasing and using cohomological descent in the $\infty$-category of sheaves as developed in \cite{GL19}.

\subsection{$\Delta S$-hypercovers}	The category $\Delta S$ has appeared in various forms throughout the literature, often studied independently with different motivations. Notable examples include Feigin and Tsygan's work on additive K-theory (\cite{FT87}), Pirashvili and Richter's reinterpretation of $\Delta S$ as the category of \emph{non-commutative sets} (\cite{PR02}), and investigations into \emph{symmetric homology} by Fiedorowicz and Loday (\cite{FL91}), as well as Krasauskas's independent development of related ideas (\cite{Krasauskas87}).

In this note, we draw upon the framework for the \emph{symmetric simplicial category} $\Delta S$ established in the context of hypercovers in \cite[\S 2]{Ban24}, which, in turn, is primarily based on the notations set up in \cite{FL91}, quoting only the results needed here. For further details, the interested reader can consult the relevant sections of \emph{loc. cit}.

\noindent\myparagraph	Observe that by \cite[Definition 2.1]{Ban24}, the simplicial scheme $\X_{\bullet}$ is a $\Delta S$-object in the category of schemes. Indeed, the symmetric group $S_{r+1}$ acts freely on $\X_r$ by permuting the factors of the fibre product in \eqref{def:XrFibreProduct}, and one immediately sees that the action maps compose with the face and degeneracy maps of the simplicial scheme $\X_r$ as per the requirements of  \cite[Definition 2.1]{Ban24}. Paired with Lemma \ref{lemma:hypercover} we obtain the following:
\begin{lemma}\label{lemma:DeltaS-scheme}
	The proper hypercover $\pi_{\bullet}:\X_{\bullet} \to \Z_{\d}(X,\P^N)$ is a  $\Delta S$-scheme augmented over $\Z_{\d}(X,\P^N)$, with the $i^{th}$ face map $\X_r\to \X_{r-1}$ given by forgetting the $i^{th}$-factor from the expression \eqref{def:XrFibreProduct}, and the degeneracy maps $\X_{r-1}\to \X_r$ given by the diagonal embeddings.
\end{lemma}

\subsubsection{Sheaves on schemes} 

For a  scheme $B$, let $\Shv(B)$ be the derived $\infty$-category of sheaves with coefficients in $\Q$-vector spaces (where $\Q$ is a field of characteristic $0$ that contains the rational numbers), equipped with Grothendieck's six functor formalism. It is a stable $\infty$-category with a $t$-structure, whose heart is the abelian category of constructible sheaves of $\Q$-vector spaces on $B$. 

For details on the construction of the sheaf theory we use in this section, the reader is directed to Gaitsgory-Lurie's work (\cite[Sections 2 and 3]{GL19}). For homological algebra on $\Shv(B)$ the reader may refer to \cite[\S 1.2, \S 1.3]{Lurie17}. For the theory of proper descent in this setup, see Liu-Zheng's \cite{LZ24}. 



\subsubsection{$\Delta S$-sheaves}\label{subsubsec:DeltaS} A $\Delta S$-sheaf on a space $B$ is a functor\[\F:\Delta S\to \Shv(B).\] As is customary, we denote such a sheaf by $\F_{\bullet}$. Note that a $\Delta S$-sheaf is naturally a $\Delta$-sheaf i.e. a cosimplicial object in $\Shv(B)$; henceforth, unless otherwise stated, notations like $\F_{\bullet}$ will always mean a $\Delta S$-sheaf.

For $\Delta S$-sheaves $\F_{\bullet}$ and  $\mathcal{G}_{\bullet}$, let $\mathrm{Hom}^{\Delta S}_{\Shv(B)}(\F_{\bullet}, \mathcal{G}_{\bullet})$ (respectively, $\mathrm{Hom}^{\Delta}_{\Shv(B)}(\F_{\bullet}, \mathcal{G}_{\bullet})$) denote the $\Delta S$-sheaf (respectively, $\Delta$-sheaf) given by maps $\F_n\to \mathcal{G}_n$ in $\Shv(B)$ that commute with the face and degeneracy maps in $\Delta S$ (respectively, $\Delta$).

Since  a $\Delta S$-sheaf $\F_{\bullet}$ is a naturally a $\Delta$-sheaf, we can consider the corresponding Moore cochain complex, which we denote by \begin{gather}
	C^*\bigg( (\F_n, d)\bigg)
\end{gather}where the differentials are, by definition, given by the alternating sum of face maps: $d=\sum_i(-1)^i d_i$ (see \cite[remarks 1.2.4.3 and 1.2.4.4]{Lurie17}). Moreover, for a $\Delta S$-sheaf $\F_{\bullet}$, we get a \emph{twisted-by-sign sheaf} for each $n$: \[\big(\F_n\otimes \sgn_{S_{n+1}}\big)^{S_{n+1}}\]where the natural action of $S_{n+1}$ on $\F_n$ is twisted by a sign
(see a similar construction in the discussion preceding \cite[Theorem 6.9]{FL91}). This supplies us with a cochain complex \begin{gather}
C^*\Big(\big(\F_n\otimes \sgn_{S_{n+1}}\big)^{S_{n+1}},d\Big) 
\end{gather}where the differential $d$ is given by the alternating sum of face maps: $d=\sum_i(-1)^i d_i$. The next  lemma focuses on relating these two.



\begin{lemma}\label{lemma:isomorphismDeltaS}
	Let $\F_{\bullet}$ be a $\Delta S$-sheaf on a scheme $B$. Then the following surjection is an isomorphism:
\begin{gather}
	C^*\bigg( (\F_n, d)\bigg)\xrightarrow{\cong} C^*\Big(\big(\F_n\otimes \sgn\big)^{S_{n+1}},d\Big)
\end{gather}

\end{lemma}

\begin{proof}[Proof of Lemma \ref{lemma:isomorphismDeltaS}]
In the case when the $\Delta S$-sheaf is in the abelian category of constructible sheaves, this is proved in \cite[Lemma 2.7]{Ban24} by a straightforward and direct adaptation of the proof of \cite[Corollary 6.17]{FL91}. 

The proof of \cite[Corollary 6.17]{FL91} readily translates to an $\infty$-categorical analogue in $\Shv(B)$. Indeed, like \cite[\S 6]{FL91}, we consider the co-bisimplicial sheaf \begin{gather}\label{def:cobisimplicial}
\mathcal{G}_{p,q}:=\Q[(S_p)^q]\otimes \F_q
\end{gather} which converges to $\pi_n(\mathrm{Hom}^{\Delta S}_{\Shv(B)}(\underline{\Q}_B, \mathcal{F}_{\bullet}))$ (the latter is equivalent to what Fiederowicz-Loday call `symmetric (co)homology', see \cite[\S 6.6]{FL91}). 

The sheaf $\mathcal{G}_{p,q}$ is a filtered object in $\Shv(B)$ in the sense of \cite[Lemma 1.2.2.4]{Lurie17}, which allows us to adapt the proof of \cite[Theorem 6.9]{FL91}: taking its horizontal filtration, and noting that $\Q$ is a field of characteristic $0$, we obtain---  \begin{equation}
E_1^{p,q}=\begin{cases}
\big(\F_p\otimes \sgn\big)^{S_{p+1}}& q=0\\0 & \text{otherwise.}
\end{cases}
\end{equation} with differentials, naturally, given by the alternating sum of face maps. 

On the other hand, $\F_{\bullet}$, now being considered a $\Delta$-sheaf, via the $\infty$-categorical Dold-Kan (\cite[Theorem 1.2.4.1]{Lurie17})
results in the cochain complex $C^*\Big(\big(\F_p, d\big)\Big)$ which converges to $\pi_n(\mathrm{Hom}^{\Delta}_{\Shv(B)}(\underline{\Q}_B, \mathcal{F}_{\bullet}))$.

By the cohomological analogue of \cite[Theorem 6.16]{FL91}, the two cohomologies are isomorphic, i.e.  $$\pi_n(\mathrm{Hom}^{\Delta}_{\Shv(B)}(\underline{\Q}_B, \mathcal{F}_{\bullet}))\cong  \pi_n(\mathrm{Hom}^{\Delta S}_{\Shv(B)}(\underline{\Q}_B, \mathcal{F}_{\bullet}))$$ for all $n$, which, in turn, proves our lemma.

\end{proof}

\section{Proof of Theorem \ref{theorem}}\label{sec:proof}
Observe that statement \ref{statement3} of Theorem \ref{theorem} immediately follows from Lemma \ref{lemma:r-separating} and Angehrn-Siu's estimate \eqref{eq:rdestimate}. 

We now prove the part of the theorem pertaining an arbitrary polarized smooth projective $Y$, i.e. statement \ref{statement1}.  And then we prove the special natural of the spectral sequence \ref{isom:result} for $Y=\P^N$.

\subsection{Proof of Theorem \ref{theorem}, \ref{statement1}}
 \begin{proof}[Proof of Statement \ref{statement1}]
By Proposition \ref{lemma:hypercoverY} (or, in the case of $Y=\P^N$, Lemma \ref{lemma:hypercover}), \begin{gather}\label{eq:cohomologicaldescent}
\pi_{\bullet}:\X_{\bullet}\to \Z_{\d}(X,Y)
\end{gather} admits cohomological descent. In other words, the unit map \begin{gather}\label{map:unit}
	\mathrm{id}_{\Z_{\d}(X,Y)}\to {\pi_{\bullet}}_*\,\pi_{\bullet}^*
\end{gather} is a natural isomorphism on $\Shv(\Z_{\d}(X,Y))$. In turn, we have:\begin{align}
\Hom(\underline{\Q}_{\Z_{\d}(X,Y)}, \underline{\Q}_{\Z_{\d}(X,Y)} )
&\cong \Hom\big(\underline{\Q}_{\Z_{\d}(X,Y)}, {\pi_{\bullet}}_*\,\pi_{\bullet}^*\underline{\Q}_{\Z_{\d}(X,Y)}\big) \\ 
&\cong C^*\Big(\Hom\big(\underline{\Q}_{\Z_{\d}(X,Y)}, {\pi_{n}}_* \underline{\Q}_{\X_n(Y)}\big)\Big)\\
& \cong C^*\Big(\Hom\big(\underline{\Q}_{\Z_{\d}(X,Y)}, ({\pi_{n}}_* \underline{\Q}_{\X_n(Y)}\otimes\sgn)^{S_{n+1}}\big)\Big),\label{isom:twistedbysgn}
\end{align} where the first isomorphism is a direct consequence of the cohomological descent in \eqref{eq:cohomologicaldescent}, the second isomorphism follows from the fact that if $f$ is a proper map, then $f^*=f^{!}$, and the third isomorphism follows from Lemma \ref{lemma:isomorphismDeltaS}.

Let us recollect the maps in \eqref{map:ifrak} and \eqref{map:jfrak}: \begin{gather}
\mor_{\d}(X,Y)\stackrel{\jfrak}{\hookrightarrow}\X_{-1}(Y)\stackrel{\ifrak}{\hookleftarrow}\Z_{\d}(X,Y).
\end{gather}where $\jfrak$ is an open immersion, and $\ifrak$ closed. So we obtain a cofibre sequence in  $\Shv(\X_{-1}(Y))$:\begin{gather}\label{seq:openclosed}
\jfrak_!\jfrak^!\underline{\Q}_{\X_{-1}(Y)} \to {\underline{\Q}}_{\X_{-1}(Y)} \to \ifrak_*\ifrak^*{\underline{\Q}}_{\X_{-1}(Y)}
\end{gather} (which, observe, is equivalent to the localization distinguished triangle in the derived category of constructible sheaves).
Now, noting that $$\ifrak^*{\underline{\Q}}_{\X_{-1}(Y)}\cong \underline{\Q}_{\Z_{\d}(X,Y)},$$ take $\Hom(\underline{\Q}_{\X_{-1}(Y)}, \text{---})$ of \eqref{seq:openclosed}, plug in \eqref{isom:twistedbysgn} and take global sections of the resulting complex to obtain the following spectral sequence: 
\begin{gather}\label{spectralsequence}
	E_1^{p,q}= \Big(H_c^q(\X_{p-1}(Y);\Q)\otimes \sgn\Big)^{S_{p}}\implies H_c^{p+q}(\mor_{\d}(X,Y);\Q)
\end{gather}where, for $p=0$, we simply let $S_0$ denote the trivial group. 
Since all our constructions are algebraic, this is a spectral sequence of Galois representations/mixed Hodge structures, as the case might be. 

Recalling Lemma \ref{lemma:Xr} on the geometry of $\X_r$, and the assumption in paragraph \ref{para:keyassumption} for the geometry of $\X_r(Y)$, we see that for all $p\leq r(\d)-1$ the spectral sequence \ref{spectralsequence} reads as: \begin{flalign}\label{isom:result}
	\Big(H_c^*(\X_{p-1}(Y);\Q)\otimes \sgn\Big)^{S_{p}} \nonumber\\ 
	\cong (H^*(\mathring{X}^{p};\Q)\otimes \sgn)^{S_p}\otimes H^*(\pic_{\d}(X);\Q)\otimes H_c^*(Y(D_{p-1});\Q)
\end{flalign}
where recall that $\mathring{X}^{p}$ denotes the space $X^p-\text{diagonals}$.

\end{proof}

\begin{remark}The reader is encouraged to compare \eqref{isom:result} with \cite[Theorem 7.1]{BG91}.\end{remark}

\subsection{Proof of Theorem \ref{theorem}, \ref{statement2}}

For $Y=\P^N$, by Lemma \ref{lemma:Xr} the terms of the spectral sequence \eqref{spectralsequence}, following \eqref{isom:result}, reads as:\begin{gather}\label{ss:P^N}
	(H^*(\mathring{X}^{p};\Q)\otimes \sgn)^{S_p}\otimes H^*(\pic_{\d}(X);\Q)\otimes H^*(\P^{(N_d-p)(N+1)-1};\Q)\nonumber\\ \cong
	(H^*(X;\Q)^{\otimes p}\otimes \sgn)^{S_p}\otimes H^*(\pic_{\d}(X);\Q)\otimes H^*(\P^{(N_d-p)(N+1)-1};\Q)
\end{gather} for all $p\leq r(\d)$. Indeed, this isomorphism follows from the simple observation that for any scheme $$(H^*(\mathring{X}^{p};\Q)\otimes \sgn)^{S_p}\cong 	(H^*(X;\Q)^{\otimes p}\otimes \sgn)^{S_p}.$$ Thus, the spectral sequence \eqref{spectralsequence} boils down, in case of $Y=\P^N$, to $E_1^{p,*} =$
\begin{align}\label{spectralsequence:P^N}
\begin{cases}
(H^*(X;\Q)^{\otimes p}\otimes \sgn)^{S_p}\otimes H^*(\pic_{\d}(X);\Q)\otimes H^*(\P^{(N_d-p)(N+1)-1};\Q) &p\leq r(\d)+1\\
\Big(H_c^q(\X_{p-1};\Q)\otimes \sgn\Big)^{S_{p}} & p>r(\d)+1
\end{cases}  
\end{align}
For a smooth projective variety $V$, the cohomology $H^i(V;\Q)$ is pure of weight $i$, which implies that the differentials between the terms of the spectral sequence \eqref{spectralsequence:P^N} vanish on the $E_2$-page and above for $p\leq r(\d)+1$. 

The result in statement \ref{statement2} follows immediately by a simple application of the Lefschetz hyperplane theorem. Indeed, 
the codimension of $\Z_{\d}(X,\P^N)$ in $\X_{-1}$ is $nN$, and likewise, a simple computation shows that the codimension of the image of $\X_r$ in $\X_{-1}$ is $\geq (r+1)nN$ for all $r\geq 0$; in turn the terms of the spectral sequence \ref{spectralsequence:P^N} for $p>r(\d)+1$ have no contribution to $H^i(\mor_{\d}(X,\P^N);\Q)$ for $i\leq r(\d)+1$, which completes the proof of Theorem \ref{theorem}, statement \ref{statement2}.

		\bibliographystyle{alpha}
	\bibliography{ConfModel}
\end{document}